\documentclass[11pt]{amsart}
\usepackage{amsmath}
\usepackage[active]{srcltx}
\usepackage{t1enc}
\usepackage[latin2]{inputenc}
\usepackage{verbatim}
\usepackage{amsmath,amsfonts,amssymb,amsthm}
\usepackage[mathcal]{eucal}
\usepackage{enumerate}
\usepackage[centertags]{amsmath}
\usepackage{graphics}

\newtheorem*{theorem}{Theorem}

\newtheorem{lemma}{Lemma}

\newtheorem*{Rodin}{Theorem R}

\newtheorem*{Kara}{Theorem K}

\def\ZT{\ensuremath{\mathbb T}}
\newcommand {\e }[1]{(\ref{#1})}

\begin{document}
\author{Ushangi Goginava and Grigori Karagulyan}
\address{U. Goginava, Department of Mathematics, Faculty of Exact and
Natural Sciences, Ivane Javakhishvili Tbilisi State University,
Chavcha\-vadze str. 1, Tbilisi 0179, Georgia}
\email{zazagoginava@gmail.com}
\address{G. A. Karagulyan, Faculty of Mathematics and Mechanics, Yerevan
State University, Alex Manoogian, 1, 0025, Yerevan, Armenia}
\email{g.karagulyan@ysu.am}
\title[Almost everywhere exponential summability]{ On almost everywhere
exponential summability of rectangular partial sums of double trigonometric
Fourier series}
\date{}
\maketitle

\begin{abstract}
In this paper we study the a.e. exponential strong summability problem for
the rectangular partial sums of double trigonometric Fourier series of the
functions from $L\log L$ .
\end{abstract}

\footnotetext{%
2010 Mathematics Subject Classification 42C10 .
}

\section{Introduction}

We denote the set of all non-negative integers by $\mathbb{N}$. Let $\mathbb{%
T}:=[-\pi ,\pi )=\mathbb{R}/2\pi $ and $\mathbb{R}:=\left( -\infty ,\infty
\right) $. Denote by $L^{1}\left( \mathbb{T}\right) $ the class of all
measurable functions $f$ on $\mathbb{R}$ that are $2\pi $-periodic and
satisfy 
\begin{equation*}
\left\Vert f\right\Vert _{1}:=\int\limits_{\mathbb{T}}\left\vert
f\right\vert <\infty .
\end{equation*}%
The Fourier series of a function $f\in L^{1}\left( \mathbb{T}\right) $ with
respect to the trigonometric system is 
\begin{equation}
\sum_{n=-\infty }^{\infty }c_{n}e^{inx},  \label{fourier}
\end{equation}%
where 
\begin{equation*}
c_{n}:=\frac{1}{2\pi }\int\limits_{\mathbb{T}}f\left( x\right) e^{-inx}dx
\end{equation*}%
are the Fourier coefficients of $f$. Denote by $S_{n}(x,f)$ the partial sums
of the Fourier series of $f$ and let 
\begin{equation*}
\sigma _{n}(x,f)=\frac{1}{n+1}\sum_{k=0}^{n}S_{k}(x,f)
\end{equation*}%
be the $(C,1)$ means of (\ref{fourier}). Fejér \cite{Fe} proved that $\sigma
_{n}(f)$ converges to $f$ uniformly for any $2\pi $-periodic continuous
function. Lebesgue in \cite{Le} established almost everywhere convergence of 
$(C,1)$ means if $f\in L^{1}(\mathbb{T})$. The strong summability problem,
i.e. the convergence of the strong means 
\begin{equation}
\frac{1}{n}\sum\limits_{k=0}^{n-1}\left\vert S_{k}\left( x,f\right) -f\left(
x\right) \right\vert ^{p},\quad x\in \mathbb{T},\quad p>0,  \label{Hp}
\end{equation}%
was first considered by Hardy and Littlewood in \cite{H-L}. They showed that
for any $f\in L^{r}(\mathbb{T})~\left( 1<r<\infty \right) $ the strong means
tend to $0$ a.e. as $n\rightarrow \infty $. The trigonometric Fourier series
of $f\in L^{1}(\mathbb{T})$ is said to be $\left( H,p\right) $-summable at $%
x\in \mathbb{T}$ if the values ({\ref{Hp}}) converge to $0$ as $n\rightarrow
\infty $. The $\left( H,p\right) $-summability problem in $L^{1}(\mathbb{T})$
has been investigated by Marcinkiewicz \cite{Ma2} for $p=2$, and later by
Zygmund \cite{Zy2} for the general case $1\leq p<\infty $.

Let $\Phi :[0,\infty )\rightarrow \lbrack 0,\infty )$, $\Phi \left( 0\right)
=0$, be a continuous increasing function. We say a series with the partial
sums $s_{n}$ strong $\Phi $-summable to a limit $s$ if 
\begin{equation*}
\lim_{n\rightarrow \infty }\frac{1}{n}\sum\limits_{k=0}^{n-1}\Phi
(|s_{k}-s|)=0.
\end{equation*}%
In \cite{Os} Oskolkov first considered the a.e strong $\Phi $-summability
problem of Fourier series with exponentially growing $\Phi $. Namely, he
proved a.e strong $\Phi $-summability of Fourier series if $\ln \Phi
(t)=O(t/\ln \ln t)$ as $t\rightarrow \infty $.

In \cite{Ro} Rodin proved

\begin{Rodin}[Rodin]
If a continuous function $\Phi :[0,\infty )\rightarrow \lbrack 0,\infty )$, $%
\Phi \left( 0\right) =0$, satisfies the condition 
\begin{equation*}
\limsup_{t\rightarrow +\infty }\frac{\ln \Phi \left( t\right) }{t}<\infty ,
\end{equation*}%
then for any $f\in L^{1}(\mathbb{T})$ the relation 
\begin{equation}
\lim\limits_{n\rightarrow \infty }\frac{1}{n}\sum\limits_{k=0}^{n-1}\Phi
(|S_{k}\left( x,f\right) -f\left( x\right) |)=0  \label{t3}
\end{equation}%
holds for a. e. $x\in \mathbb{T}$.
\end{Rodin}

Karagulyan \cite{Ka1, Ka} proved that the exponential growth in Rodin's
theorem is optimal. Moreover, it was proved

\begin{Kara}[Karagulyan]
If a continuous increasing function $\Phi :[0,\infty )\rightarrow \lbrack
0,\infty ),\Phi \left( 0\right) =0$, satisfies the condition%
\begin{equation*}
\limsup_{t\rightarrow +\infty }\frac{\ln \Phi \left( t\right) }{t}=\infty ,
\end{equation*}%
then there exists a function $f\in L^{1}(\mathbb{T})$, for which the relation%
\begin{equation*}
\limsup_{n\rightarrow \infty }\frac{1}{n}\sum\limits_{k=0}^{n-1}\Phi \left(
\left\vert S_{k}\left( x,f\right) \right\vert \right) =\infty 
\end{equation*}%
holds everywhere on $\mathbb{T}$.
\end{Kara}

In this paper we study the exponential summability problem for the
rectangular partial sums of double Fourier series. Let $f\in L^{1}(\mathbb{T}%
^{2})$ be a function with Fourier series 
\begin{equation}
\sum_{m,n=-\infty }^{\infty }c_{nm} e^{i(mx+ny)},  \label{DF}
\end{equation}%
where 
\begin{equation*}
c_{nm} =\frac{1}{4\pi ^{2}}\iint\limits_{\mathbb{T}%
^{2}}f(x_{1},x_{2})e^{-i(mx_{1}+nx_{2})}dx_{1}dx_{2}
\end{equation*}%
are the Fourier coefficients of the function $f$. The rectangular partial
sums of (\ref{DF}) are defined by 
\begin{equation*}
S_{MN}\left(f\right)=S_{MN}\left( x_{1},x_{2},f\right)
=\sum_{m=-M}^{M}\sum_{n=-N}^{N}c_{nm}e^{i(mx_{1}+nx_{2})}.
\end{equation*}%
We denote by $L\log L\left( \mathbb{T}^{2}\right) $ the class of measurable
functions $f$, with 
\begin{equation*}
\iint\limits_{\mathbb{T}^{2}}|f|\log^{+} |f|<\infty ,
\end{equation*}%
where $\log ^{+}u:=\mathbb{I}_{(1,\infty )}\log u$, $u>0$. For the
rectangular partial sums of two-dimensional trigonometric Fourier series
Jessen, Marcinkiewicz and Zygmund \cite{JMZ} has proved for any $f\in L\log
L\left( \mathbb{T}^{2}\right) $ that 
\begin{equation*}
\lim\limits_{n,m\rightarrow \infty }\frac{1}{nm}\sum\limits_{i=0}^{n-1}\sum%
\limits_{j=0}^{m-1}\left( S_{ij}\left( x_{1},x_{2},f\right) -f\left(
x_{1},x_{2}\right) \right) =0
\end{equation*}%
for a. e. $\left( x_{1},x_{2}\right) \in \mathbb{T}^{2}$. They also showed
that for every non-negative function $\omega :[0,\infty )\rightarrow \lbrack
0,\infty )$ satisfying $\omega(t) \uparrow \infty$, $\omega \left( t\right)
\left( \log^+ t\right) ^{-1}\rightarrow 0$ as $t\rightarrow \infty $, there
exists a function $f$ such that $\left\vert f\right\vert \omega \left(
\left\vert f\right\vert \right) \in L^{1}\left( \mathbb{T}^{2}\right) $ and
the $\left(C,1,1\right) $ means of double Fourier series of $f$ diverge a.e..

The two dimensional a.e. strong rectangular $(H,p)$-summability, i.e. the
relation 
\begin{equation*}
\lim\limits_{n,m\rightarrow \infty }\frac{1}{nm}\sum\limits_{i=0}^{n-1}\sum%
\limits_{j=0}^{m-1}\left\vert S_{ij}\left( x_{1},x_{2},f\right) -f\left(
x_{1},x_{2}\right) \right\vert ^{p}=0\text { a.e. }
\end{equation*}%
was proved by Gogoladze \cite{Gog} for $f\in L\log L\left( \mathbb{T}%
^{2}\right) $. These results show that in two dimensional case the optimal
class of functions for $\left( C,1,1\right) $ summability and strong
summability coincide. That is the class of functions $L\log L\left( \mathbb{T%
}^{2}\right) $.

We prove the following

\begin{theorem}
\label{exp} If  a continuous increasing function $\Phi :[0,\infty
)\rightarrow \lbrack 0,\infty )$, $\Phi \left( 0\right) =0$, satisfies the
condition 
\begin{equation}
\limsup_{t\rightarrow +\infty }\frac{\ln \Phi \left( t\right) }{\sqrt{t/\ln
\ln t}}<\infty ,  \label{t4}
\end{equation}%
then for any $f\in L\log L\left( \mathbb{T}^{2}\right) $ the relation 
\begin{equation}
\lim\limits_{n,m\rightarrow \infty }\frac{1}{nm}\sum\limits_{i=0}^{n-1}\sum%
\limits_{j=0}^{m-1}\Phi \left( \left\vert S_{ij}\left( x_{1},x_{2},f\right)
-f(x_{1},x_{2})\right\vert \right) =0  \label{t1}
\end{equation}%
holds for a. e. $\left( x_{1},x_{2}\right) \in \mathbb{T}^{2}$.
\end{theorem}

As a corollary of this result we get the Gogoladze \cite{Gog} theorem on
a.e. $H^p$-summability of double Fourier series. From Jessen, Marcinkiewicz
and Zygmund \cite{JMZ} theorem it follows that the class $ L\log L\left( 
\mathbb{T}^{2}\right) $ in our theorem is necessary  in the context of
strong summability question. That is, it is not possible  to give a larger
convergence space than $L\log L\left( \mathbb{T} ^{2}\right) $. Our method
of proof do not allow to get (\ref{t1}) under the weaker condition  
\begin{equation}  \label{t6}
\limsup_{t\rightarrow +\infty }\frac{\ln \Phi \left( t\right) }{\sqrt t}%
<\infty.
\end{equation}
There is a conjecture that \e {t6} is the optimal bound of $\Phi$ ensuring
a.e. rectangular strong summability (\ref{t1}) for every function $f\in
L\log L\left( \mathbb{T}^{2}\right) $.

The results on strong summability and approximation by trigonometric Fourier
series have been extended for several other orthogonal systems, see Schipp 
\cite{Sch1,Sch2,Sch3}, Leindler \cite{Le1,Le2,Le3,Le4}, Totik \cite%
{tot,To1,To2,To3}, Goginava, Gogoladze \cite{GGCA1,GGSMH}, Goginava,
Gogoladze, Karagulyan \cite{GGKCA2}, Gat, Goginava, Karagulyan \cite%
{GGKAM,GGKJMAA}, Weisz \cite{WeC2}-\cite{WeJFS}.

\section{Auxiliary lemmas}

The notation $a\lesssim b$ will stand for $a<c\cdot b$, where $c>0$ is an
absolute constant. We shall write $a\sim b$ if the relations $a\lesssim b$
and $b\lesssim a$ hold at the same time. Everywhere below $q>1$ will be used
as the conjugate of $p>1$, that is $1/p+1/q=1$. $[a]$ denotes the integer
part of $a\in\mathbb{R}  $.

The maximal function of a function $f\in L^{1}\left( \mathbb{T}\right) $ is
defined by 
\begin{equation*}
Mf\left( x\right) :=\sup\limits_{I:\,x\in I\subset \mathbb{T}}\frac{1}{%
\left\vert I\right\vert }\int\limits_{I}\left\vert f\left( y\right)
\right\vert dy,
\end{equation*}%
where $I$ is an open interval. The following one dimensional operators
introduced by Gabisonia \cite{Gab} are significant tools in the
investigations of strong summability problems: 
\begin{equation*}
G_{p}^{\left( n\right) }f(x):=\left( \sum\limits_{k=1}^{\left[ n\pi \right]
}\left( \frac{n}{k}\int\limits_{\frac{k-1}{n}}^{\frac{k}{n}}\left\vert
f\left( x+t\right) \right\vert +\left\vert f\left( x-t\right) \right\vert
dt\right) ^{q}\right) ^{1/q},
\end{equation*}%
\begin{equation*}
G_{p}f(x):=\sup\limits_{n\in \mathbb{N}}G_{p}^{\left( n\right) }f(x).
\end{equation*}%
Oskolkov's following lemma plays key role in the proof of the basic lemma.

\begin{lemma}[Oskolkov, \protect\cite{Os}]
For any family of pairwise disjoint intervals $\Delta _{k}\subset \mathbb{T}$
with centers $c_{k}$ it holds the inequality 
\begin{equation}\label{a1}
\left\vert \left\{ x\in \mathbb{T}:\,\sup_{p>1}\frac{\sum_{j}\left( 
\frac{|\Delta _{j}|}{|x-c_{j}|+|\Delta _{j}|}\right) ^{q}}{p\ln \ln (p+2)}%
>\lambda \right\} \right\vert \lesssim \exp (-c\lambda ),\,\lambda >0,
\end{equation}%
where $c>0$ is an absolute constant.
\end{lemma}
One can easily check that
\begin{equation*}
\sup_{p>1}\frac{\left(\sum_{j}\left( 
	\frac{|\Delta _{j}|}{|x-c_{j}|+|\Delta _{j}|}\right)^{q}\right)^{1/q}}{p\ln \ln (p+2)}\lesssim \left\{1,\sup_{p>1}\frac{\sum_{j}\left( 
	\frac{|\Delta _{j}|}{|x-c_{j}|+|\Delta _{j}|}\right) ^{q}}{p\ln \ln (p+2)}\right\}.
\end{equation*}
Combining this with \e {a1}, we get
\begin{equation}\label{a2}
\int_{\ZT}\sup_{p>1}\frac{\left(\sum_{j}\left( 
	\frac{|\Delta _{j}|}{|x-c_{j}|+|\Delta _{j}|}\right)^{q}\right)^{1/q}}{p\ln \ln (p+2)}\lesssim 1.
\end{equation}
\begin{lemma}
\label{mainlemma}If $f\in L^{1}\left( \mathbb{T}\right) $, then%
\begin{equation}
\left\vert \left\{ x\in \mathbb{T}:\sup\limits_{p>1}\frac{G_pf(x) }{%
p\ln\ln (p+2)}> \lambda \right\} \right\vert \lesssim \left( \frac{1}{%
\lambda }\left\Vert f\right\Vert _{1}\right) ^{1/2},\quad \lambda>0.
\label{est}
\end{equation}
\end{lemma}

\begin{proof}
It is enough to prove the same estimate for the modified operators 
\begin{equation}
G_{p}^{\prime }f(x):=\sup_{n\in \mathbb{N}}\left( \sum\limits_{k=1}^{\left[
n\pi \right] }\left( \frac{n}{k}\int\limits_{\frac{k-1}{n}}^{\frac{k}{n}%
}\left\vert f\left( x+t\right) \right\vert dt\right) ^{q}\right) ^{1/q}.
\label{GanOp2}
\end{equation}%
Using the Calderon-Zygmund lemma, for the maximal function we get the
relation 
\begin{equation}
R_{\lambda }:=\left\{ x\in \mathbb{T}:Mf\left( x\right) >\sqrt{\lambda }%
\right\} =\bigcup\limits_{k=0}^{\infty }\Delta _{k},\quad \lambda >0,
\label{rep}
\end{equation}%
where $\Delta _{k}\subset \mathbb{T}$ are disjoint open intervals such that%
\begin{align}
& \sqrt{\lambda }\leq \frac{1}{\left\vert \Delta _{k}\right\vert }%
\int\limits_{\Delta _{k}}\left\vert f\left( t\right) \right\vert dt\leq 2%
\sqrt{\lambda },  \label{two-side} \\
& \left\vert R_{\lambda }\right\vert \leq \frac{1}{\sqrt{\lambda }}%
\left\Vert f\right\Vert _{1}.  \label{one-est}
\end{align}%
Denote $\delta _{k}^{n}:=\left[ \left( k-1\right) /n,k/n\right] $ and $%
\delta _{k}^{n}\left( x\right) :=x+\delta _{k}^{n}$. Separating the terms in
the sum (\ref{GanOp2}) with $k$ satisfying $\delta _{k}^{n}\left( x\right)
\subset R_{\lambda }$, we get%
\begin{align}
G_{p}^{\prime }f(x)& \leq \sup\limits_{n\in \mathbb{N}}\left(
\sum\limits_{k:\delta _{k}^{n}\left( x\right) \subset R_{\lambda }}\left( 
\frac{n}{k}\int\limits_{\frac{k-1}{n}}^{\frac{k}{n}}\left\vert f\left(
x+t\right) \right\vert dt\right) ^{q}\right) ^{1/q}  \label{gabsplit} \\
& \qquad +\sup\limits_{n\in \mathbb{N}}\left( \sum\limits_{k:\delta
_{k}^{n}\left( x\right) \not\subset R_{\lambda }}\left( \frac{n}{k}%
\int\limits_{\frac{k-1}{n}}^{\frac{k}{n}}\left\vert f\left( x+t\right)
\right\vert dt\right) ^{q}\right) ^{1/q}  \notag \\
& :=I+II.  \notag
\end{align}

From the definition of $R_{\lambda }$ in the case $\delta _{k}^{n}\left(
x\right) \not\subset R_{\lambda }$ it follows that 
\begin{equation*}
n\int\limits_{\frac{k-1}{n}}^{\frac{k}{n}}\left\vert f\left( x+t\right)
\right\vert dt\leq \sqrt{\lambda }.
\end{equation*}%
Thus we conclude 
\begin{equation}
II\leq \sqrt{\lambda }\left( \sum\limits_{k=1}^{\infty }\frac{1}{k^{q}}%
\right) ^{1/q}\lesssim \sqrt{\lambda }\left( \frac{1}{q-1}\right)
^{1/q}\lesssim p\sqrt{\lambda }.  \label{II}
\end{equation}%
Given $x\in \mathbb{T}$ set 
\begin{equation*}
k_{i}(x)=\left\{ 
\begin{array}{llr}
\min \left\{ k:\delta _{k}^{n}\left( x\right) \subset \Delta _{i}\right\}  & %
\hbox{ if }\left\{ k:\delta _{k}^{n}\left( x\right) \subset \Delta
_{i}\right\} \neq \varnothing , &  \\ 
\infty  & \hbox { if }\left\{ k:\delta _{k}^{n}\left( x\right) \subset
\Delta _{i}\right\} =\varnothing . & 
\end{array}%
\right. 
\end{equation*}%
Denote $\widetilde{R}_{\lambda }:=\bigcup\limits_{k=1}^{\infty }3\Delta _{k}$
and take an arbitrary point $x\in \mathbb{T}\backslash \widetilde{R}%
_{\lambda }$. One can easily check that if $k_{i}(x)\neq \infty $, then 
\begin{equation*}
\Delta _{i}\ni \frac{k_{i}\left( x\right) }{n}\sim \left\vert
x-c_{i}\right\vert ,
\end{equation*}%
where $c_{i}$ is the center of the interval $\Delta _{i}$. Thus for any $%
x\notin \widetilde{R}_{\lambda }$ we obtain 
\begin{align}
I& =\sup_{n\in \mathbb{N}}\left( \sum\limits_{i=1}^{\infty
}\sum\limits_{k:\delta _{k}^{n}\left( x\right) \subset \Delta _{i}}\left( 
\frac{n}{k}\int\limits_{\delta _{k}^{n}\left( x\right) }\left\vert f\left(
t\right) \right\vert dt\right) ^{q}\right) ^{1/q}  \label{I1} \\
& \leq \sup_{n\in \mathbb{N}}\left( \sum\limits_{i=1}^{\infty }\left(
\sum\limits_{k:\delta _{k}^{n}\left( x\right) \subset \Delta _{i}}\frac{n}{k}%
\int\limits_{\delta _{k}^{n}\left( x\right) }\left\vert f\left( t\right)
\right\vert dt\right) ^{q}\right) ^{1/q}  \notag \\
& \le \sup_{n\in \mathbb{N}}\left( \sum\limits_{i=1}^{\infty }\left( 
\frac{n\left\vert \Delta _{i}\right\vert }{k_{i}\left( x\right) }\frac{1}{%
\left\vert \Delta _{i}\right\vert }\int\limits_{\Delta _{i}}\left\vert
f\left( t\right) \right\vert dt\right) ^{q}\right) ^{1/q}  \notag \\
& \lesssim \sqrt{\lambda }\sup_{n}\left( \sum\limits_{i=1}^{\infty }\left( 
\frac{n\left\vert \Delta _{i}\right\vert }{k_{i}\left( x\right) }\right)
^{q}\right) ^{1/q}  \notag \\
& \lesssim \sqrt{\lambda }\left( \sum\limits_{i=1}^{\infty }\left( \frac{%
\left\vert \Delta _{i}\right\vert }{\left\vert x-c_{i}\right\vert +|\Delta
_{i}|}\right) ^{q}\right) ^{1/q},\quad x\notin \widetilde{R}_{\lambda }. 
\notag
\end{align}

Using Chebyshev's inequality, from \e {a2}, (\ref{II}) and (\ref{I1}) it follows that 
\begin{align*}
&\left\vert \left\{ x\in \mathbb{T}\backslash \widetilde{R}_{\lambda
}:\sup\limits_{p>1}\frac{G_{p}^{\prime }f(x) }{p\ln\ln (p+2)}> \lambda
\right\} \right\vert \\
&\qquad\lesssim \left\vert \left\{ x\in \mathbb{T}\backslash \widetilde{R}%
_{\lambda }:\sqrt{\lambda} \left(1+\sup\limits_{p>1}\frac{%
\left(\sum_j\left(\frac{|\Delta_j|}{|x-c_j|+|\Delta_j|}\right)^q\right)^{1/q}%
}{p\ln\ln (p+2)}\right)\geq c\lambda \right\} \right\vert  \notag \\
&\qquad\lesssim \frac{1}{\sqrt{\lambda}}\int\limits_{\mathbb{T}%
}\sup\limits_{p>1}\frac{\left(\sum_j\left(\frac{|\Delta_j|}{%
|x-c_j|+|\Delta_j|}\right)^q\right)^{1/q}}{p\ln\ln (p+2)}dx  \notag \\
&\qquad \lesssim \frac{1}{\sqrt{\lambda}},  \notag
\end{align*}
for an appropriate absolute constant $c>0$. Applying homogeneity, one can get 
\begin{equation}
\left\vert \left\{ x\in \mathbb{T}\backslash \widetilde{R}_{\lambda
}:\sup\limits_{p>1}\frac{G_{p}^{\prime }f(x) }{p\ln\ln (p+2)}> \lambda
\right\} \right\vert\lesssim \left(\frac{\|f\|_1}{\lambda}%
\right)^{1/2},\quad \lambda>0.  \label{I2}
\end{equation}

Consequently, from (\ref{one-est})-(\ref{I2}) we get%
\begin{align*}
&\left\vert \left\{ x\in \mathbb{T}:\sup\limits_{p>1}\frac{G_{p}^{\prime
}f(x) }{p\ln\ln (p+2) }>\lambda \right\} \right\vert \\
& \qquad \le \left\vert \left\{ x\in \mathbb{T}\backslash \widetilde{R}%
_{\lambda }:\sup\limits_{p>1}\frac{G_{p}^{\prime }f(x) }{p\ln\ln (p+2)}%
>\lambda \right\} \right\vert +|\tilde R_\lambda| \\
&\qquad \lesssim \left( \frac{\left\Vert f\right\Vert_1 }{\lambda }\right)
^{1/2}+ \frac{\left\Vert f\right\Vert_1 }{\sqrt \lambda }.
\end{align*}
Again using homogeneity, we obtain \e {est}.
\end{proof}

We will need the following estimations.

\begin{lemma}[Gabisonia, \protect\cite{Gab}]
\label{gabest} If $p>1$ and $f\in L^{1}(\mathbb{T}^{2})$, then 
\begin{equation}\label{a4}
\left( \frac{1}{n}\sum\limits_{j=0}^{n-1}\left\vert S_{j}(x,f)\right\vert
^{p}\right) ^{1/p}\lesssim G_{p}^{\left( n\right) }f(x).
\end{equation}
\end{lemma}

\begin{lemma}[Schipp, \protect\cite{Sch}]
\label{sch1}If $f\in L^1(\mathbb{T} ^2)$, then 
\begin{equation}\label{a5}
\left( \frac{1}{n}\sum\limits_{j=0}^{n-1}\left\vert S_{j}(x,f) \right\vert
^{p}\right) ^{1/p} \lesssim pG_2f(x) .
\end{equation}
\end{lemma}

Rodin \cite{Ro} proved the weak $(1,1)$-type estimate for the operators $%
G_{p}f(x)$ with a fixed $p>1$. From this fact, applying a standard
argument, one can derive

\begin{lemma}[Rodin, \protect\cite{Ro}]
\label{sch2}Let $f\in L\log L\left( \mathbb{T}\right) $. Then%
\begin{equation*}
\left\Vert G_2\left( f\right) \right\Vert _{1}\lesssim 1+\int\limits_{%
\mathbb{T}}\left\vert f\right\vert \log \left\vert f\right\vert .
\end{equation*}
\end{lemma}

For any function $f\in L^{1}(\mathbb{T}^{2})$ define 
\begin{eqnarray*}
G_{p,1}(x_{1},x_{2};f) &=&G_{p}f_{x_{2}}(x_{1}),\quad
G_{p,2}(x_{1},x_{2};f)=G_{p}f_{x_{1}}(x_{2}), \\
G_{p,1}^{\left( n\right) }(x_{1},x_{2};f) &=&G_{p}^{\left( n\right)
}f_{x_{2}}(x_{1}),\quad G_{p,2}^{\left( n\right)
}(x_{1},x_{2};f)=G_{p}^{\left( n\right) }f_{x_{1}}(x_{2}),
\end{eqnarray*}%
where $f_{x_{2}}(\cdot )=f(\cdot ,x_{2})$ and $f_{x_{1}}(\cdot
)=f(x_{1},\cdot )$ are considered as functions on $x_{1}$ and $x_{2}$
respectively. Similarly one dimensional partial sums of $f(x_{1},x_{2})$
with respect to each variables will be denoted by 
\begin{equation*}
S_{n,1}(x_{1},x_{2},f)=S_{n}(x_{1},f_{x_{2}}),\quad
S_{n,2}(x_{1},x_{2},f)=S_{n}(x_{2},f_{x_{1}}).
\end{equation*}

\begin{lemma}
\label{main}If $f\in L\log L(\mathbb{T}^{2})$, then 
\begin{align*}
&\left\vert \left\{ \sup\limits_{p>1}\sup\limits_{n,m\in \mathbb{N}}%
\frac{\left( \frac{1}{nm}\sum\limits_{i=0}^{n-1}\sum\limits_{j=0}^{m-1}\left%
\vert S_{i,j}\left( x_{1},x_{2},f\right) \right\vert ^{p}\right) ^{1/p}}{%
p^{2}\ln\ln (p+2)}>\lambda \right\} \right\vert \\
&\qquad\qquad\lesssim \left( \frac{1}{\lambda }\left( 1+\iint\limits_{%
\mathbb{T}^{2}}\left\vert f\right\vert \log ^{+}\left\vert f\right\vert
\right) \right) ^{1/2},\quad \lambda>0.
\end{align*}
\end{lemma}

\begin{proof}
Using \e {a4}, \e {a5} and generalized Minkowsi's inequality, we get 
\begin{align*}
\frac{1}{nm}\sum\limits_{i=0}^{n-1}\sum\limits_{j=0}^{m-1}\left\vert
S_{i,j}\left( x_{1},x_{2},f\right) \right\vert ^{p}& =\frac{1}{nm}%
\sum\limits_{i=0}^{n-1}\sum\limits_{j=0}^{m-1}\left\vert S_{i,1}\left(
x_{1},x_{2},S_{j,2}\left( f\right) \right) \right\vert ^{p} \\
& \leq \frac{1}{m}\sum\limits_{j=0}^{m-1}\left( G_{p,1}^{\left( n\right)
}\left( x_{1},x_{2},\left\vert S_{j,2}\left( f\right) \right\vert \right)
\right) ^{p} \\
& \leq \left( G_{p,1}^{\left( n\right) }\left( x_{1},x_{2},\left( \frac{1}{m}%
\sum\limits_{j=0}^{m-1}\left\vert S_{j,2}\left( f\right) \right\vert
^{p}\right) ^{1/p}\right) \right) ^{p} \\
& \leq \left( G_{p,1}\left( x_{1},x_{2},\left( \frac{1}{m}%
\sum\limits_{j=0}^{m-1}\left\vert S_{j,2}\left( f\right) \right\vert
^{p}\right) ^{1/p}\right) \right) ^{p} \\
& \lesssim p^{p}\left( G_{p,1}\left( x_{1},x_{2},G_{2,2}\left( f\right)
\right) \right) ^{p}.
\end{align*}%
Hence we obtain 
\begin{align*}
\Omega & =\left\{ \left( x_{1},x_{2}\right) \in \mathbb{T}%
^{2}:\sup\limits_{p>1}\sup\limits_{n,m\in \mathbb{N}}\frac{\left( \frac{1%
}{nm}\sum\limits_{i=0}^{n-1}\sum\limits_{j=0}^{m-1}\left\vert S_{i,j}\left(
x_{1},x_{2},f\right) \right\vert ^{p}\right) ^{1/p}}{p^{2}\ln \ln (p+2)}%
>\lambda \right\}  \\
& \subset \left\{ \left( x_{1},x_{2}\right) \in \mathbb{T}%
^{2}:\sup\limits_{p>1}\frac{G_{p,1}\left( x_{1},x_{2},G_{2,2}\left(
f\right) \right) }{p\ln \ln (p+2)}>\lambda \right\},
\end{align*}%
then, applying Lemma \ref{mainlemma} and \ref{sch2}, we conclude 
\begin{align*}
\left\vert \Omega \right\vert & =\int\limits_{\mathbb{T}^{2}}\mathbb{I}%
_{\Omega }\left( x_{1},x_{2}\right) dx_{1}dx_{2}=\int\limits_{\mathbb{T}%
}dx_{2}\int\limits_{\mathbb{T}}\mathbb{I}_{\Omega }\left( x_{1},x_{2}\right)
dx_{1} \\
& \lesssim \int\limits_{\mathbb{T}}\left( \frac{1}{\lambda }\int\limits_{%
\mathbb{T}}G_{2,2}\left( x_{1},x_{2},f\right) dx_{1}\right) ^{1/2}dx_{2} \\
& \lesssim \int\limits_{\mathbb{T}}\left[ \frac{1}{\lambda }\left(
1+\int\limits_{\mathbb{T}}\left\vert f\left( x_{1},x_{2}\right) \right\vert
\log ^{+}\left\vert f\left( x_{1},x_{2}\right) \right\vert dx_{1}\right) %
\right] ^{1/2}dx_{2} \\
& \lesssim \left[ \frac{1}{\lambda }\left( 1+\int\limits_{\mathbb{T}%
^{2}}\left\vert f\left( x_{1},x_{2}\right) \right\vert \log ^{+}\left\vert
f\left( x_{1},x_{2}\right) \right\vert dx_{1}dx_{2}\right) \right] ^{1/2}.
\end{align*}
Lemma is proved.
\end{proof}

\section{\protect\bigskip Proof of Theorem \protect\ref{exp}}

Let $L_{M}:=L_{M}\left(\mathbb{T} ^2\right)$ be Orlicz space of
functions on $\mathbb{T} ^2$ generated by the Young function $M(t)=t\log ^{+}t$.
It is known that $L_{M}$ is a Banach space with respect to the Luxemburg
norm 
\begin{equation*}
\left\Vert f\right\Vert _{M}:=\inf \left\{ \lambda :\lambda
>0,\int\limits_{X}M\left( \frac{\left\vert f\right\vert }{\lambda }\right)
\leq 1\right\} <\infty .
\end{equation*}%
According to a theorem from (\cite{KrRu}, Chap. 2, theorem 9.5) we have%
\begin{equation*}
0,5\left( 1+\int\limits_{\mathbb{T}^{2}}M\left( \left\vert f\right\vert
\right) \right) \leq \left\Vert f\right\Vert _{M}\leq 1+\int\limits_{\mathbb{%
\ T}^{2}}M\left( \left\vert f\right\vert \right)
\end{equation*}%
provided $\left\Vert f\right\Vert _{M}=1$. Hence from Lemma \ref{main} we
conclude 
\begin{equation}  \label{exp-est}
\left\vert \left\{ \sup\limits_{p>1}\sup\limits_{n,m\in \mathbb{N}}\frac{
\left( \frac{1}{nm}\sum\limits_{i=0}^{n-1}\sum\limits_{j=0}^{m-1}\left\vert
S_{i,j}\left(f\right) \right\vert ^{p}\right) ^{1/p}}{ p^2\log\log (p+2)}%
>\lambda \right\} \right\vert \lesssim \left( \frac{ \left\Vert f\right\Vert
_M}{\lambda }\right) ^{1/2}.
\end{equation}%
Indeed, at first we deduce the case of $\left\Vert f\right\Vert _M=1$, then
using a homogeneity argument, we get the inequality in the general case.

\begin{proof}[Proof of Theorem]
First we shall prove that for any $f\in L\log L(\mathbb{T}^{2})$ the relation 
\begin{equation}
\lim_{n,m\rightarrow \infty }\sup_{p>1}\frac{\left( \frac{1}{nm}%
\sum\limits_{i=0}^{n-1}\sum\limits_{j=0}^{m-1}|S_{i,j}(f)-f|^{p}\right)
^{1/p}}{p^{2}\ln \ln (p+2)}=0\text{ a.e. }.  \label{t2}
\end{equation}%
Observe that \e {t2} trivially holds for the double trigonometric
polynomials. Indeed, let $T$ be a trigonometric polynomial of degree $%
(s_{1},s_{2})$. Then we have 
\begin{align*}
& S_{i,j}\left( T\right) -T=0,\quad i\geq s_{1},j\geq s_{2}, \\
& S_{i,j}\left( T\right) -T=S_{s_{1},j}\left( T\right) -T,\quad i\geq
s_{1},\,0\leq j<s_{2}, \\
& S_{i,j}\left( T\right) -T=S_{i,s_{2}}\left( T\right) -T,\quad 0\leq
i<s_{1},\,j\geq s_{2}.
\end{align*}%
Thus for integers $n> s_{1}$ and $m> s_{2}$ we have 
\begin{align*}
\frac{1}{nm}\sum_{i=0}^{n-1}& \sum_{j=0}^{m-1}|S_{i,j}(T)-T|^{p} \\
& =\frac{1}{n}\sum_{i=0}^{s_{1}-1}\frac{1}{m}%
\sum_{j=s_{2}}^{m-1}|S_{i,s_{2}}(T)-T|^{p}+\frac{1}{m}\sum_{j=0}^{s_{2}-1}%
\frac{1}{n}\sum_{i=s_{1}}^{n-1}|S_{s_{1},j}(T)-T|^{p} \\
& \qquad +\frac{1}{nm}\sum_{i=0}^{s_{1}-1}%
\sum_{j=0}^{s_{2}-1}|S_{i,j}(T)-T|^{p} \\
& \leq \frac{1}{n}\sum_{i=0}^{s_{1}-1}|S_{i,s_{2}}(T)-T|^{p}+\frac{1}{m}%
\sum_{j=0}^{s_{2}-1}|S_{s_{1},j}(T)-T|^{p} \\
& \qquad +\frac{1}{nm}\sum_{j=0}^{s_{2}-1}%
\sum_{i=0}^{s_{2}-1}|S_{i,j}(T)-T|^{p} \\
& \leq \frac{c_{1}}{n}+\frac{c_{2}}{m},
\end{align*}%
where $c_{1}$ and $c_{2}$ are constants depended on $T$. Thus (\ref{t2})
holds if $f=T$. To prove the general case it is enough to show that the set 
\begin{equation*}
G_{\lambda }=\left\{ \limsup_{n,m\rightarrow \infty }\,\sup_{p>1}\frac{%
\left( \frac{1}{nm}\sum\limits_{i=0}^{n-1}\sum%
\limits_{j=0}^{m-1}|S_{i,j}(f)-f|^{p}\right) ^{1/p}}{p^{2}\ln \ln (p+2)}%
>\lambda \right\} 
\end{equation*}%
has measure zero for any $\lambda >0$. Since $M\left( t\right) $ satisfies
the $\Delta _{2}$-condition, the function $f$ can be approximated by a trigonometric polynomial $T$(see 
\cite{KrRu}), that is
\begin{equation*}
\left\Vert f-T\right\Vert _{M}<\varepsilon,\quad \left\Vert f-T\right\Vert _{L^1}<\varepsilon
\end{equation*}%
Since (\ref{t2}) holds for $T$, applying (\ref{exp-est}), one can obtain 
\begin{align*}
|G_{\lambda }|& =\left\vert \left\{ \limsup_{n,m\rightarrow \infty
}\,\sup_{p>1}\frac{\left( \frac{1}{nm}\sum\limits_{i=0}^{n-1}\sum%
\limits_{j=0}^{m-1}|S_{i,j}(f-T)-(f-T)|^{p}\right) ^{1/p}}{p^{2}\ln \ln (p+2)%
}>\lambda \right\} \right\vert  \\
& \leq \left\vert \left\{ \sup_{n,m\in \mathbb{N}}\,\sup_{p>1}\frac{%
\left( \frac{1}{nm}\sum\limits_{i=0}^{n-1}\sum%
\limits_{j=0}^{m-1}|S_{i,j}(f-T)|^{p}\right) ^{1/p}}{p^{2}\ln \ln (p+2)}%
>\lambda /2\right\} \right\vert  \\
& \qquad +\left\vert \left\{ \sup_{p>1}\frac{|f-T|}{p^{2}\ln \ln (p+2)}%
>\lambda /2\right\} \right\vert  \\
& \lesssim \left( \frac{\left\Vert f-T\right\Vert _{M}}{\lambda }\right)
^{1/2}+\frac{\|f\|_{L^1}}{\lambda} \\
& \leq \left( \frac{\varepsilon }{\lambda }\right) ^{1/2}+\frac{\varepsilon }{\lambda }.
\end{align*}%
Since $\varepsilon >0$ can be taken arbitrarily small, we conclude that $%
\left\vert G_{\lambda }\right\vert =0$ for any $\lambda >0$ and so \e {t2}
holds. To prove \e {t1} observe that 
\begin{align}
u(s)&=\exp \left( \sqrt{\frac{s}{\ln \ln (s+2)}}\right) \label{t5}\\
&\leq v(s)=\sum_{k=1}^{\infty }\left( \frac{d}{k}\sqrt{\frac{s}{\ln \ln (k+2)}}%
\right) ^{k},\,s>1,\nonumber  
\end{align}%
for some absolute constant $d$. Indeed, if $s\geq 1$, then one can check that 
\begin{equation*}
1<\sqrt{\frac{s}{\ln \ln (s+2)}}<k(s)=\left[ \sqrt{\frac{s}{\ln \ln (s+2)}}%
\right] +1<2\sqrt{\frac{s}{\ln \ln (s+2)}},
\end{equation*}%
and therefore for enough bigger $d$ we we will have
\begin{align*}
v(s)& \geq \left( \frac{d}{k(s)}\sqrt{\frac{s}{\ln \ln (k(s)+2)}}\right)
^{k(s)} \\
& >\left( \frac{d}{2}\sqrt{\frac{\ln \ln (s+2)}{\ln \ln (k(s)+2)}}\right)
^{k(s)}>e^{k(s)} \\
& \geq u(s)
\end{align*}%
and so \e {t5}. If the function $\Phi $ satisfies \e {t4}, then one
can check that 
	\begin{equation*}
	\Phi(s)\le \exp \left(\sqrt{\frac{A\cdot s}{\ln\ln (A\cdot s+2)}}\right)=u(As),\quad s>S,
	\end{equation*}
for some positive numbers $A>1,S>1$. Consider the functions 
\begin{align*}
& \varphi _{i,j}(f)=S_{i,j}(f)-f, \\
& \varphi _{i,j}^{\ast }(f)=\left\{ 
\begin{array}{lc}
\varphi _{i,j}(f) & \hbox{ if }|\varphi _{i,j}(f)|\leq S, \\ 
0 & \hbox { if }|\varphi _{i,j}(f)|>S, \\ 
& 
\end{array}%
\right.  \\
& \varphi _{i,j}^{\ast \ast }(f)=\varphi _{i,j}(f)-\varphi _{i,j}^{\ast }(f).
\end{align*}%
From \e {t5} and the definition of $\Phi$ it follows that 
\begin{align*}
\frac{1}{nm}\sum_{i=0}^{n-1}\sum_{j=0}^{m-1}\Phi (|\varphi _{i,j}(f)|)& =%
\frac{1}{nm}\sum_{i=0}^{n-1}\sum_{j=0}^{m-1}\Phi (|\varphi _{i,j}^{\ast
}(f)|)+\frac{1}{nm}\sum_{i=0}^{n-1}\sum_{j=0}^{m-1}\Phi (|\varphi
_{i,j}^{\ast \ast }(f)|) \\
& \leq \frac{1}{nm}\sum_{i=0}^{n-1}\sum_{j=0}^{m-1}\Phi (|\varphi
_{i,j}^{\ast }(f)|)+v\left(A|\varphi
_{i,j}^{\ast \ast }(f)|\right) \\
& = \frac{1}{nm}\sum_{i=0}^{n-1}\sum_{j=0}^{m-1}\Phi (|\varphi
_{i,j}^{\ast }(f)|) \\
& \qquad +\sum_{k=1}^{\infty }\frac{1}{nm}\sum\limits_{i=0}^{n-1}\sum%
\limits_{j=0}^{m-1}\left( \frac{d}{k}\sqrt{\frac{A\cdot |\varphi
_{i,j}^{\ast \ast }(f)|}{\ln \ln (k+2)}}\right) ^{k} \\
& \leq \frac{1}{nm}\sum_{i=0}^{n-1}\sum_{j=0}^{m-1}\Phi (|\varphi
_{i,j}^{\ast }(f)|) \\
& \quad +\sum_{k=1}^{\infty }(d\sqrt{A})^{k}\left( \sup_{p>1/2}\frac{%
\left( \frac{1}{nm}\sum\limits_{i=0}^{n-1}\sum\limits_{j=0}^{m-1}|\varphi
_{i,j}(f)|^{p}\right) ^{1/p}}{4p^{2}\ln \ln (2p+2)}\right) ^{\frac{k}{2}}.
\end{align*}%
The second term of the last expression tends to zero almost everywhere,
since according to \e {t2} we have 
\begin{align*}
\limsup_{n,m\rightarrow \infty }\sup_{p>1/2}& \frac{\left( \frac{1}{nm}%
\sum\limits_{i=0}^{n-1}\sum\limits_{j=0}^{m-1}|\varphi _{i,j}(f)|^{p}\right)
^{1/p}}{4p^{2}\ln \ln (2p+2)} \\
& \leq \lim_{n,m\rightarrow \infty }\sup_{p>1}\frac{\left( \frac{1}{nm}%
\sum\limits_{i=0}^{n-1}\sum\limits_{j=0}^{m-1}|S_{i,j}(f)-f|^{p}\right)
^{1/p}}{p^{2}\ln \ln (p+2)}=0\text{ a.e. }
\end{align*}%
Hence, to prove \e {t1} it is enough to show the same for the first term.
From \e {t2} and Chebyshev's inequality it follows that 
\begin{align*}
r_{n,m}(x_{1},x_{2})& =\frac{\#\{i,j\in \mathbb{N}:\,0\leq i<n,\,0\leq
j<m,\,\varphi _{i,j}(x_{1},x_{2})>\varepsilon \}}{nm} \\
& \leq \frac{1}{\varepsilon }\cdot \frac{1}{nm}\sum\limits_{i=0}^{n-1}\sum%
\limits_{j=0}^{m-1}|\varphi _{i,j}(x_{1},x_{2},f)|\rightarrow 0\text{ a.e.},
\end{align*}%
where $\#C$ denotes the cardinality of a finite set $C$. Thus for a.e. $%
(x_{1},x_{2})\in \mathbb{T}^{2}$ we get 
\begin{align*}
\limsup_{n,m\rightarrow \infty }\frac{1}{nm}&
\sum_{i=0}^{n-1}\sum_{j=0}^{m-1}\Phi (|\varphi _{i,j}^{\ast
}(x_{1},x_{2},f)|) \\
& \leq \limsup_{n,m\rightarrow \infty }(r_{n,m}(x_{1},x_{2})\Phi
(S)+(1-r_{n,m}(x_{1},x_{2}))\Phi (\varepsilon )) \\
& =\Phi (\varepsilon )\text{ a.e.}.
\end{align*}%
Since $\varepsilon >0$ can be taken arbitrary small we get 
\begin{equation*}
\lim_{n,m\rightarrow \infty }\frac{1}{nm}\sum_{i=0}^{n-1}\sum_{j=0}^{m-1}%
\Phi (|\varphi _{i,j}^{\ast }(x_{1},x_{2},f)|)=0\text{ a.e.}
\end{equation*}%
and so \e {t1}.
\end{proof}

\end{document}